\documentclass[reqno]{amsart}
\usepackage{amssymb}
\newtheorem{theorem}{Theorem}[section]
\newtheorem{proposition}[theorem]{Proposition}
\newtheorem{lemma}[theorem]{Lemma}

\theoremstyle{definition}
\newtheorem{definition}[theorem]{Definition}

\newtheorem{remark}[theorem]{Remark}
\allowdisplaybreaks
\numberwithin{equation}{section}
\begin{document}

\title[transformations and values of $p$-adic hypergeometric functions]
{Certain transformations and values of $p$-adic hypergeometric functions}


%    author one information
 %\author[short version for running head]{}
 \author{Sulakashna}
\address{Department of Mathematics, Indian Institute of Technology Guwahati, North Guwahati, Guwahati-781039, Assam, INDIA}
\curraddr{}
\email{sulakash@iitg.ac.in}
\author{Rupam Barman}
\address{Department of Mathematics, Indian Institute of Technology Guwahati, North Guwahati, Guwahati-781039, Assam, INDIA}
\curraddr{}
\email{rupam@iitg.ac.in}

\thanks{}

%    author two information

%    \subjclass is required.
\date{20th April 2022, version-1}
%\thanks{Acknowledement: We thank the referee for his/her valuable comments.}
%    Abstract is required.
\begin{abstract} We prove two transformations for the $p$-adic hypergeometric functions which can be described as $p$-adic analogues of a Euler's transformation and a
transformation of Clausen. We first evaluate certain character sums, and then relate them to the $p$-adic hypergeometric functions to deduce the transformations. 
We use a character sum identity proved by Ahlgren, Ono, and Penniston to deduce the $p$-adic Clausen's transformation. We also deduce special values of certain $p$-adic hypergeometric functions.
\end{abstract}
\maketitle
\section{Introduction and statement of results}
For a non-negative integer $r$, and $a_i, b_i\in\mathbb{C}$ with $b_i\notin\{\ldots, -3,-2,-1\}$, 
the classical hypergeometric series ${_{r+1}}F_{r}$ is defined by
\begin{align}
{_{r+1}}F_{r}\left(\begin{array}{cccc}
                   a_0, & a_1, & \ldots, & a_{r} \\
                    & b_1, & \ldots, & b_r
                 \end{array}| \lambda
\right):=\sum_{k=0}^{\infty}\frac{(a_0)_k\cdots (a_{r})_k}{(b_1)_k\cdots(b_r)_k}\cdot\frac{\lambda^k}{k!},\notag
\end{align}
where, for a complex number $a$, the rising factorial or the Pochhammer symbol $(a)_k$ is defined as $(a)_0=1$ and $(a)_k=a(a+1)\cdots (a+k-1), ~k\geq 1$. It is well-known that the classical hypergeometric series ${_{r+1}}F_{r}$ converges absolutely for $|\lambda|<1$.
\par In recent times, many authors have studied certain finite field analogues of the classical hypergeometric series. It seems that hypergeometric functions over a finite field first appeared in Koblitz's work \cite{kob2}. 
There are other definitions of hypergeometric functions over finite fields. For example, see the works of Greene \cite{greene, greene2}, Katz \cite{katz}, McCarthy \cite{mccarthy3}, Fuselier et al. \cite{FL}, and Otsubo \cite{noriyuki}. 
Some of the biggest motivations for studying finite field hypergeometric functions have been their connections with Fourier coefficients and eigenvalues of modular forms and with counting points on certain kinds of algebraic varieties. 
For example, see \cite{ahlgren, BK, BK1, evans-mod, frechette, fuselier, Fuselier-McCarthy, koike, lennon, lennon2, mccarthy4, mortenson, mc-papanikolas, ono, salerno, vega}. 
\par 
Results involving finite field hypergeometric functions are often restricted to primes in certain congruence classes to facilitate the existence of characters of specific orders. To overcome these restrictions, McCarthy \cite{mccarthy1, mccarthy2} 
defined a function in terms of quotients of the $p$-adic gamma function. For an odd prime $p$, let $\mathbb{F}_q$ denote the finite field with $q$ elements, where $q=p^r, r\geq 1$. Let $\mathbb{Z}_p$ denote the ring of $p$-adic integers. 
Let $\Gamma_p(.)$ denote the Morita's $p$-adic gamma function, and let $\omega$ denote the 
Teichm\"{u}ller character of $\mathbb{F}_q$. We denote by $\overline{\omega}$ the inverse of $\omega$. For $x \in \mathbb{Q}$, we let $\lfloor x\rfloor$ denote the greatest integer less than
or equal to $x$ and $\langle x\rangle$ denote the fractional part of $x$, i.e., $x-\lfloor x\rfloor$, satisfying $0\leq\langle x\rangle<1$.
We now recall the definition of McCarthy's $p$-adic hypergeometric function $_{n}G_{n}[\cdots]_q$.
\begin{definition}\cite[Definition 5.1]{mccarthy2} \label{defin1}
	Let $p$ be an odd prime and $q=p^r$, $r\geq 1$. Let $t \in \mathbb{F}_q$.
	For positive integer $n$ and $1\leq k\leq n$, let $a_k$, $b_k$ $\in \mathbb{Q}\cap \mathbb{Z}_p$.
	Then the function $_{n}G_{n}[\cdots]_q$ is defined by
	\begin{align}
	&_nG_n\left[\begin{array}{cccc}
	a_1, & a_2, & \ldots, & a_n \\
	b_1, & b_2, & \ldots, & b_n
	\end{array}|t
	\right]_q\notag\\
	&\hspace{1cm}:=\frac{-1}{q-1}\sum_{a=0}^{q-2}(-1)^{an}~~\overline{\omega}^a(t)
	\prod\limits_{k=1}^n\prod\limits_{i=0}^{r-1}(-p)^{-\lfloor \langle a_kp^i \rangle-\frac{ap^i}{q-1} \rfloor -\lfloor\langle -b_kp^i \rangle +\frac{ap^i}{q-1}\rfloor}\notag\\
	&\hspace{2cm} \times \frac{\Gamma_p(\langle (a_k-\frac{a}{q-1})p^i\rangle)}{\Gamma_p(\langle a_kp^i \rangle)}
	\frac{\Gamma_p(\langle (-b_k+\frac{a}{q-1})p^i \rangle)}{\Gamma_p(\langle -b_kp^i \rangle)}.\notag
	\end{align}
\end{definition}
The function $_{n}G_{n}[\cdots]_q$ extends finite field hypergeometric functions to the $p$-adic setting, and often allows results involving finite field hypergeometric functions to be extended to a wider class of primes \cite{BS2, BS1, BSM, BS3, BS4, mccarthy2}. 
It is a non-trivial and important problem to find $p$-adic analogues of identities satisfied by the classical hypergeometric series. Only a few such results are known till date. For example, see \cite{BS3, Fuselier-McCarthy, NS}.
\par In this article, we prove two transformations for the function $_{n}G_{n}[\cdots]_q$ which can be described as $p$-adic analogues of a Euler's transformation and a
transformation of Clausen. The Euler's transformation \cite[p. 10]{slater} is given by
\begin{align}\label{Euler}
{_2F}_{1}\left(\begin{array}{cc}
a, & b \\
&c
\end{array}|x\right)=(1-x)^{c-a-b}{_2F}_{1}\left(\begin{array}{cc}
c-a, & c-b \\
&c
\end{array}|x\right).
\end{align}
Greene proved a finite field analogue of \eqref{Euler} in \cite[Theorem 4.4 (iv)]{greene}. Finding a $p$-adic analogue of \eqref{Euler} for any $a, b, c \in \mathbb{Q}\cap \mathbb{Z}_p$ seems to be a difficult problem. 
In the following theorem, we prove a transformation for the $p$-adic hypergeometric function which can be described as a $p$-adic analogue of \eqref{Euler} for certain particular values of $a, b$, and $c$. Let $\varphi$ denote the quadratic character on $\mathbb{F}_q$.
\begin{theorem}\label{MT-1}
Let $p\geq 5$ be a prime and $q=p^r$, $r\geq 1$. Then, for $x\in\mathbb{F}_q$ such that $x\neq 0, 1$, we have
\begin{align}\label{eq1-MT-1}
{_{2}G}_{2}\left[\begin{array}{cc}
           \frac{1}{3}, & \frac{2}{3} \\
           0, & \frac{1}{2}
         \end{array}|\frac{1}{x}\right]_q= \varphi(1-x)\cdot  {_2G}_{2}\left[\begin{array}{cc}
           \frac{1}{6}, & \frac{5}{6} \\
           0, & \frac{1}{2}
         \end{array}|\frac{1}{x}\right]_q.
\end{align}
Furthermore, 
\begin{align}\label{eq2-MT-1}
{_{2}G}_{2}\left[\begin{array}{cc}
           \frac{1}{3}, & \frac{2}{3} \\
           0, & \frac{1}{2}
         \end{array}|1\right]_q= \varphi(3)\cdot  {_2G}_{2}\left[\begin{array}{cc}
           \frac{1}{6}, & \frac{5}{6} \\
           0, & \frac{1}{2}
         \end{array}|1\right]_q.
\end{align}
\end{theorem}
Next, we classify the zeros of the function $\displaystyle { _{2}G}_{2}\left[\begin{array}{cc}
\frac{1}{3}, & \frac{2}{3} \\
0, & \frac{1}{2}
\end{array}|\frac{1}{x}\right]_q$ in the following theorem. 
\begin{theorem}\label{MT-3}
Let $p\geq 5$ be a prime and $q=p^r$, $r\geq 1$. Let $x\in\mathbb{F}_q$ be such that $x\neq 0, 1$. Then
	\begin{align*}
	{_{2}G}_{2}\left[\begin{array}{cc}
	\frac{1}{3}, & \frac{2}{3} \\
	0, & \frac{1}{2}
	\end{array}|\frac{1}{x}\right]_q= {_2G}_{2}\left[\begin{array}{cc}
	\frac{1}{6}, & \frac{5}{6} \\
	0, & \frac{1}{2}
	\end{array}|\frac{1}{x}\right]_q=0
	\end{align*}
	if and only if $\varphi(3x(1-x))=-1$.
\end{theorem}
The following transformation for classical hypergeometric series is a special case
of Clausen's famous classical identity \cite[p. 86, Eq. (4)]{bailey}.
\begin{align}\label{clausen}
{_3}F_2\left(\begin{array}{ccc}
           \frac{1}{2}, & \frac{1}{2}, &\frac{1}{2} \\
           &1,&1
         \end{array}|x\right)=(1-x)^{-1/2}~{_2}F_1\left(\begin{array}{cc}
           \frac{1}{4}, & \frac{3}{4} \\
           &1
         \end{array}|\frac{x}{x-1}\right)^2.
\end{align}
A finite field analogue of \eqref{clausen} was studied by Greene \cite[p. 94, Prop. 6.14]{greene2}. In \cite{EG}, Evans and Greene also gave a finite field analogue of the
Clausen's classical identity. The $p$-adic analogue of \eqref{clausen} is proven in \cite[Theorem 1.3]{BS3}, but over $\mathbb{F}_{p}$. In the following theorem, we prove a $p$-adic analogue of \eqref{clausen} over a finite field $\mathbb{F}_q$. 
We use a character sum identity of Ahlgren, Ono and Penniston \cite[Theorem 2.1]{ahlgren1} to deduce the following transformation over $\mathbb{F}_q$.
\begin{theorem}\label{MT-2}
Let $p$ be an odd prime and $q=p^{r}$, $r\geq 1$. Then, for $x\in\mathbb{F}_q$ such that $x\neq 0, 1$, we have
\begin{align*}
{_3G}_{3}\left[\begin{array}{ccc}
\frac{1}{2}, & \frac{1}{2}, & \frac{1}{2} \\
0, & 0, & 0
\end{array}|\frac{1}{x}
\right]_q &=\varphi(1-x)\cdot 
{_2G}_{2}\left[\begin{array}{cc}
           \frac{1}{4}, & \frac{3}{4} \\
           0, & 0
         \end{array}|\frac{x-1}{x}\right]_q^2-q\cdot \varphi(1-x).
\end{align*}
\end{theorem}
\section{Preliminaries}
For an odd prime $p$, let $\mathbb{F}_q$ denote the finite field with $q$ elements, where $q=p^r, r\geq 1$. Let $\widehat{\mathbb{F}_q^{\times}}$ be the group of all multiplicative
characters on $\mathbb{F}_q^{\times}$. We extend the domain of each $\chi\in \widehat{\mathbb{F}_q^{\times}}$ to $\mathbb{F}_q$ by setting $\chi(0)=0$
including the trivial character $\varepsilon$. We will make use of the following orthogonality relation for characters:
\begin{align}\label{eq-2}
\sum_{\chi\in \widehat{\mathbb{F}_q^{\times}}} \chi(x)= \left\{
\begin{array}{ll}
q-1, & \hbox{if $x=1$;} \\
0, & \hbox{otherwise.}
\end{array}
\right.
\end{align}
For multiplicative characters $A$ and $B$ on $\mathbb{F}_q$,
the binomial coefficient ${A \choose B}$ is defined by
\begin{align}\label{eq-0}
{A \choose B}:=\frac{B(-1)}{q}J(A,\overline{B})=\frac{B(-1)}{q}\sum_{x \in \mathbb{F}_q}A(x)\overline{B}(1-x),
\end{align}
where $J(A, B)$ denotes the Jacobi sum and $\overline{B}$ is the character inverse of $B$. It is easy to see that the Jacobi sum satisfies the following identity:
\begin{align}\label{eq-1}
J(A,B)=A(-1)J(A,\overline{AB}).
\end{align}
Let $\delta$ denote the function on multiplicative characters defined by
\begin{align*}\delta(A)=\left\{
\begin{array}{ll}
1, & \hbox{if $A$ is the trivial character;} \\
0, & \hbox{otherwise.}
\end{array}
\right.
\end{align*}
We recall the following properties of the binomial coefficients from \cite{greene}:
\begin{align}\label{eq-4}
{A\choose \varepsilon}={A\choose A}=\frac{-1}{q}+\frac{q-1}{q}\delta(A).
\end{align}
\par
Let $\mathbb{Z}_p$ and $\mathbb{Q}_p$ denote the ring of $p$-adic integers and the field of $p$-adic numbers, respectively.
Let $\overline{\mathbb{Q}_p}$ be the algebraic closure of $\mathbb{Q}_p$ and $\mathbb{C}_p$ the completion of $\overline{\mathbb{Q}_p}$.
Let $\mathbb{Z}_q$ be the ring of integers in the unique unramified extension of $\mathbb{Q}_p$ with residue field $\mathbb{F}_q$.
We know that $\chi\in \widehat{\mathbb{F}_q^{\times}}$ takes values in $\mu_{q-1}$, where $\mu_{q-1}$ is the group of
$(q-1)$-th roots of unity in $\mathbb{C}^{\times}$. Since $\mathbb{Z}_q^{\times}$ contains all $(q-1)$-th roots of unity,
we can consider multiplicative characters on $\mathbb{F}_q^\times$
to be maps $\chi: \mathbb{F}_q^{\times} \rightarrow \mathbb{Z}_q^{\times}$.
Let $\omega: \mathbb{F}_q^\times \rightarrow \mathbb{Z}_q^{\times}$ be the Teichm\"{u}ller character.
For $a\in\mathbb{F}_q^\times$, the value $\omega(a)$ is just the $(q-1)$-th root of unity in $\mathbb{Z}_q$ such that $\omega(a)\equiv a \pmod{p}$.
\par Next, we introduce Gauss sum and recall some of its elementary properties. For further details, see \cite{evans}. Let $\zeta_p$ be a fixed primitive $p$-th root of unity
in $\overline{\mathbb{Q}_p}$. The trace map $\text{tr}: \mathbb{F}_q \rightarrow \mathbb{F}_p$ is given by
\begin{align}
\text{tr}(\alpha)=\alpha + \alpha^p + \alpha^{p^2}+ \cdots + \alpha^{p^{r-1}}.\notag
\end{align}
%Then the additive character
%$\theta: \mathbb{F}_q \rightarrow \mathbb{Q}_p(\zeta_p)$ is defined by
%\begin{align}
%\theta(\alpha)=\zeta_p^{\text{tr}(\alpha)}.\notag
%\end{align}
For $\chi \in \widehat{\mathbb{F}_q^\times}$, the \emph{Gauss sum} is defined by
\begin{align}
g(\chi):=\sum\limits_{x\in \mathbb{F}_q}\chi(x)\zeta_p^{\text{tr}(x)}.\notag
\end{align}
\begin{lemma}\emph{(\cite[Eq. 1.12]{greene}).}\label{lemma2_1}
For $\chi \in \widehat{\mathbb{F}_q^\times}$, we have
$$g(\chi)g(\overline{\chi})=q\cdot \chi(-1)-(q-1)\delta(\chi).$$
\end{lemma}
The following lemma gives a relation between Jacobi and Gauss sums.
\begin{lemma}\emph{(\cite[Eq. 1.14]{greene}).}\label{lemma2_2} For $A,B\in\widehat{\mathbb{F}_q^{\times}}$ we have
\begin{align}
J(A,B)=\frac{g(A)g(B)}{g(AB)}+(q-1)B(-1)\delta(AB).\notag
\end{align}
\end{lemma}
\begin{theorem}\emph{(\cite[Davenport-Hasse Relation]{evans}).}\label{thm2_2}
Let $m$ be a positive integer and let $q=p^r$ be a prime power such that $q\equiv 1 \pmod{m}$. For multiplicative characters
$\chi, \psi \in \widehat{\mathbb{F}_q^\times}$, we have
\begin{align}
\prod\limits_{\chi^m=\varepsilon}g(\chi \psi)=-g(\psi^m)\psi(m^{-m})\prod\limits_{\chi^m=\varepsilon}g(\chi).\notag
\end{align}
\end{theorem}
Now, we recall the $p$-adic gamma function. For further details, see \cite{kob}.
For a positive integer $n$,
the $p$-adic gamma function $\Gamma_p(n)$ is defined as
\begin{align}
\Gamma_p(n):=(-1)^n\prod\limits_{0<j<n,p\nmid j}j\notag
\end{align}
and one extends it to all $x\in\mathbb{Z}_p$ by setting $\Gamma_p(0):=1$ and
\begin{align}
\Gamma_p(x):=\lim_{x_n\rightarrow x}\Gamma_p(x_n)\notag
\end{align}
for $x\neq0$, where $x_n$ runs through any sequence of positive integers $p$-adically approaching $x$.
This limit exists, is independent of how $x_n$ approaches $x$,
and determines a continuous function on $\mathbb{Z}_p$ with values in $\mathbb{Z}_p^{\times}$.
Let $\pi \in \mathbb{C}_p$ be the fixed root of $x^{p-1} + p=0$ which satisfies
$\pi \equiv \zeta_p-1 \pmod{(\zeta_p-1)^2}$. Then the Gross-Koblitz formula relates Gauss sums and the $p$-adic gamma function as follows.
\begin{theorem}\emph{(\cite[Gross-Koblitz]{gross}).}\label{thm2_3} For $a\in \mathbb{Z}$ and $q=p^r, r\geq 1$, we have
\begin{align}
g(\overline{\omega}^a)=-\pi^{(p-1)\sum\limits_{i=0}^{r-1}\langle\frac{ap^i}{q-1} \rangle}\prod\limits_{i=0}^{r-1}\Gamma_p\left(\left\langle \frac{ap^i}{q-1} \right\rangle\right).\notag
\end{align}
\end{theorem}
The following lemmas relate certain products of values of the $p$-adic gamma function.
\begin{lemma}\emph{(\cite[Lemma 3.1]{BS1}).}\label{lemma3_1}
Let $p$ be a prime and $q=p^r, r\geq 1$. For $0\leq a\leq q-2$ and $t\geq 1$ with $p\nmid t$, we have
\begin{align}
\omega(t^{-ta})\prod\limits_{i=0}^{r-1}\Gamma_p\left(\left\langle\frac{-tp^ia}{q-1}\right\rangle\right)
\prod\limits_{h=1}^{t-1}\Gamma_p\left(\left\langle \frac{hp^i}{t}\right\rangle\right)
=\prod\limits_{i=0}^{r-1}\prod\limits_{h=0}^{t-1}\Gamma_p\left(\left\langle\frac{p^i(1+h)}{t}-\frac{p^ia}{q-1}\right\rangle \right).\notag
\end{align}
\end{lemma}
\begin{lemma}\emph{(\cite[Lemma 3.2]{BS1}).}\label{lemma3_2}
Let $p$ be a prime and $q=p^r, r\geq 1$. For $0\leq a\leq q-2$ and $t\geq 1$ with $p\nmid t$, we have
\begin{align*}
\omega(t^{ta})\prod\limits_{i=0}^{r-1}\Gamma_p\left(\left\langle\frac{tp^ia}{q-1}\right\rangle\right)
\prod\limits_{h=1}^{t-1}\Gamma_p\left(\left\langle \frac{hp^i}{t}\right\rangle\right)
=\prod\limits_{i=0}^{r-1}\prod\limits_{h=0}^{t-1}\Gamma_p\left(\left\langle\frac{p^i h}{t}+\frac{p^ia}{q-1}\right\rangle \right).\notag
\end{align*}
\end{lemma}
We prove the following lemmas which will be used to prove our main results.
\begin{lemma}\label{lemma-3.2}
Let $p$ be an odd prime and $q=p^r, r\geq 1$. Then, for $0\leq a\leq q-2$ such that $a\neq \frac{q-1}{2}$ and $0\leq i\leq r-1$, we have
\begin{align}\label{eq-5}
&-2\left\lfloor\frac{2ap^i}{q-1}\right\rfloor -\left\lfloor\frac{-6ap^i}{q-1}\right\rfloor+\left\lfloor\frac{ap^i}{q-1}\right\rfloor +\left\lfloor\frac{-3ap^i}{q-1}\right\rfloor\nonumber\\
&=-\left\lfloor\left\langle\frac{p^i}{6}\right\rangle-\frac{ap^i}{q-1}\right\rfloor-\left\lfloor\left\langle\frac{5p^i}{6}\right\rangle-\frac{ap^i}{q-1}\right\rfloor-\left\lfloor\left\langle\frac{p^i}{2}\right\rangle+\frac{ap^i}{q-1}\right\rfloor-\left\lfloor\frac{ap^i}{q-1}\right\rfloor.
\end{align}
\end{lemma}
\begin{proof}
If $a=0$, then it readily follows that \eqref{eq-5} is true. For $a\neq 0$, we write $\left\lfloor\frac{-6ap^i}{q-1}\right\rfloor=6k+s$, where $k, s \in \mathbb{Z}$ satisfying $0\leq s\leq 5$.
Then, we have
\begin{align}\label{eq-6}
6k+s\leq\frac{-6ap^i}{q-1}< 6k+s+1.
\end{align}
If $p^i\equiv 1\pmod{6}$, then \eqref{eq-6} yields
\begin{align*}
\left\lfloor\frac{ap^i}{q-1}\right\rfloor = -k-1,
\end{align*}
\begin{align}\label{eq-7}
\left\lfloor\frac{2ap^i}{q-1}\right\rfloor=\left\{
                                              \begin{array}{ll}
                                                -2k-1, & \hbox{if $s=0,1,2$;} \\
                                                -2k-2, & \hbox{if $s=3,4,5$,}
                                              \end{array}
                                            \right.
\end{align}
\begin{align}\label{eq-8}
\left\lfloor\frac{-3ap^i}{q-1}\right\rfloor=\left\{
                                              \begin{array}{ll}
                                                3k, & \hbox{if $s=0,1$;} \\
                                                3k+1, & \hbox{if $s=2,3$;}\\
                                                 3k+2, & \hbox{if $s=4,5$,}\\
                                              \end{array}
                                            \right.
\end{align}
\begin{align*}
\left\lfloor\left\langle\frac{p^i}{6}\right\rangle-\frac{ap^i}{q-1}\right\rfloor=\left\{
                                                                         \begin{array}{ll}
                                                                           k, & \hbox{if $s=0,1,2,3,4$;} \\
                                                                           k+1, & \hbox{if $s=5$,}
                                                                         \end{array}
                                                                       \right. 
\end{align*}
\begin{align}
\left\lfloor\left\langle\frac{5p^i}{6}\right\rangle-\frac{ap^i}{q-1}\right\rfloor=\left\{
                                                                         \begin{array}{ll}
                                                                           k, & \hbox{if $s=0$;} \\
                                                                           k+1, & \hbox{if $s=1,2,3,4,5$,}
                                                                         \end{array}
                                                                       \right.\notag
\end{align}
and
\begin{align}\label{eq-9}
\left\lfloor\left\langle\frac{p^i}{2}\right\rangle+\frac{ap^i}{q-1}\right\rfloor=\left\{
                                                                         \begin{array}{ll}
                                                                           -k, & \hbox{if $s=0,1,2$;} \\
                                                                           -k-1, & \hbox{if $s=3,4,5$.}
                                                                         \end{array}
                                                                       \right.
\end{align}
We note that in \eqref{eq-7} and \eqref{eq-9}, we need to assume that $a\neq \frac{q-1}{2}$. Putting the above values for different values of $s$ we readily obtain \eqref{eq-5}. Proof of \eqref{eq-5} goes along similar lines when $p^i\equiv 5\pmod{6}$. 
\end{proof}
\begin{lemma}\label{lemma-3.3}
Let $p$ be an odd prime and $q=p^r, r\geq 1$. Then, for $0< a\leq q-2$ and $0\leq i\leq r-1$, we have
\begin{align}\label{eq-10}
&-\left\lfloor\frac{2ap^i}{q-1}\right\rfloor-\left\lfloor\frac{-3ap^i}{q-1}\right\rfloor\nonumber\\ 
&=1-\left\lfloor\left\langle\frac{p^i}{3}\right\rangle-\frac{ap^i}{q-1}\right\rfloor-\left\lfloor\left\langle\frac{2p^i}{3}\right\rangle-\frac{ap^i}{q-1}\right\rfloor-\left\lfloor\left\langle\frac{p^i}{2}\right\rangle+\frac{ap^i}{q-1}\right\rfloor.
\end{align}
\end{lemma}
\begin{proof}
If $a= \frac{q-1}{2}$, then it readily follows that \eqref{eq-10} is true. For $a\neq \frac{q-1}{2}$, we write	
$\left\lfloor\frac{-6ap^i}{q-1}\right\rfloor=6k+s$, where $k, s \in \mathbb{Z}$ satisfying $0\leq s\leq 5$.
Then, we have
\begin{align}\label{new-eq-6}
	6k+s\leq\frac{-6ap^i}{q-1}< 6k+s+1.
\end{align}
If $p^i\equiv 1\pmod{6}$, then \eqref{new-eq-6} yields
\begin{align*}
\left\lfloor\left\langle\frac{p^i}{3}\right\rangle-\frac{ap^i}{q-1}\right\rfloor=\left\{
                                                                         \begin{array}{ll}
                                                                           k, & \hbox{if $s=0,1,2,3$;} \\
                                                                           k+1, & \hbox{if $s=4,5$,}
                                                                         \end{array}
                                                                       \right. 
\end{align*}
and
\begin{align*}
\left\lfloor\left\langle\frac{2p^i}{3}\right\rangle-\frac{ap^i}{q-1}\right\rfloor=\left\{
                                                                         \begin{array}{ll}
                                                                           k, & \hbox{if $s=0,1$;} \\
                                                                           k+1, & \hbox{if $s=2,3,4,5$.}
                                                                         \end{array}
                                                                       \right. 
\end{align*}
Putting the above values and using \eqref{eq-7}, \eqref{eq-8}, and \eqref{eq-9} for different values of $s$ we readily obtain \eqref{eq-10}. The proof of \eqref{eq-10} goes along similar lines when $p^i\equiv 5\pmod{6}$.
\end{proof}
\begin{lemma}
Let $p$ be an odd prime and $q=p^{r}, r\geq 1$. For $0<j\leq q-2$, we have
\begin{align}\label{eq-12}
\prod_{i=0}^{r-1} \Gamma_{p}(\langle(1-\frac{j}{q-1})p^{i}\rangle)\Gamma_{p}(\langle\frac{jp^{i}}{q-1}\rangle) = (-1)^r \overline{\omega}^{j}(-1).
\end{align}
For $0\leq j\leq q-2$ such that $j\neq \frac{q-1}{2}$, we have
\begin{align}\label{eq-13}
\prod_{i=0}^{r-1} \frac{\Gamma_{p}(\langle(\frac{1}{2}-\frac{j}{q-1})p^{i}\rangle)\Gamma_{p}(\langle(\frac{1}{2}+\frac{j}{q-1})p^{i}\rangle)}{\Gamma_{p}\langle\frac{p^{i}}{2}\rangle\Gamma_{p}\langle\frac{p^{i}}{2}\rangle} = \overline{\omega}^{j}(-1).
\end{align}
\end{lemma}
\begin{proof}
	The proof readily follows using Gross-Koblitz formula and Lemma \ref{lemma2_1}.
\end{proof}
 \section{Proofs of the main results}
We first prove two propositions which enable us to express certain character sums in terms of the $p$-adic hypergeometric functions.
\begin{proposition}\label{prop-1}
Let $p\geq 5$ be a prime and $q=p^r, r\geq 1$. For $x\in\mathbb{F}_q^{\times}$, we have
\begin{align}
\frac{1}{q(q-1)}\sum_{\chi\in\widehat{\mathbb{F}_q^{\times}}}
{g(\chi)}{g(\chi^{2})}{g(\overline{\chi}^{3})}\chi\left(\frac{-27}{4x}\right)=\frac{1}{q}+{_{2}G}_{2}\left[\begin{array}{cc}
           \frac{1}{3}, & \frac{2}{3} \\
           0, & \frac{1}{2}
         \end{array}|\frac{1}{x}\right]_q.\notag
\end{align}
\end{proposition}
\begin{proof}
Let $T$ be a generator of the cyclic group $\widehat{\mathbb{F}_q^{\times}}$.  Then we have
\begin{align}
A&:= \frac{1}{q(q-1)}\sum_{\chi\in\widehat{\mathbb{F}_q^{\times}}}
{g(\chi)}{g(\chi^{2})}{g(\overline{\chi}^{3})}\chi\left(\frac{-27}{4x}\right)\nonumber\\
&=\frac{1}{q(q-1)} \sum_{a=0}^{q-2}
{g(T^{a})}{g(T^{2a})}{g(T^{-3a})}T^{a}\left(\frac{-27}{4x}\right).\notag
\end{align}
Now, taking $T=\overline{\omega}$ and then applying Gross-Koblitz formula we deduce that
\begin{align}
A&= \frac{1}{q(q-1)}\sum_{a=0}^{q-2} 
{g(\overline{\omega}^a)}{g(\overline{\omega}^{2a})}{g(\overline{\omega}^{-3a}})\overline{\omega}^a\left(\frac{-27}{4x}\right)\notag\\
&=\displaystyle  -\frac{1}{q(q-1)}\sum_{a=0}^{q-2}\overline{\omega}^a\left(\frac{-27}{4x}\right)\prod_{i=0}^{r-1}(-p)^{\alpha_{i,a}}\Gamma_p(\langle\frac{ap^{i}}{q-1}\rangle) \Gamma_p(\langle\frac{2ap^{i}}{q-1}\rangle) \Gamma_p(\langle\frac{-3ap^{i}}{q-1}\rangle)\notag, 
\end{align}
where $\alpha_{i,a}=\langle\frac{ap^{i}}{q-1}\rangle +\langle\frac{2ap^{i}}{q-1}\rangle+\langle\frac{-3ap^{i}}{q-1}\rangle$. Taking out the term for $a=0$ gives
\begin{align*}
A&= -\frac{1}{q(q-1)}\sum_{a=1}^{q-2}\overline{\omega}^a\left(\frac{-27}{4x}\right)\prod_{i=0}^{r-1}(-p)^{\alpha_{i,a}}\Gamma_p (\langle\frac{ap^{i}}{q-1}\rangle) \Gamma_p (\langle\frac{2ap^{i}}{q-1}\rangle) \Gamma_p (\langle\frac{-3ap^{i}}{q-1}\rangle) \nonumber\\
&\hspace{.4cm}- \frac{1}{q(q-1)}.
\end{align*}
Using Lemma \ref{lemma3_1} for $t=3$, Lemma \ref{lemma3_2} for $t=2$ and Lemma \ref{lemma-3.3}, we deduce that
\begin{align*}
A&= -\frac{1}{q(q-1)}\sum_{a=1}^{q-2}\overline{\omega}^a\left(\frac{-1}{x}\right)\prod_{i=0}^{r-1}(-p)^{1+ s_{i,a}}\frac{\Gamma_p(\langle\frac{ap^{i}}{q-1}\rangle) \Gamma_p(\langle\frac{p^{i}}{2}+\frac{ap^{i}}{q-1}\rangle) 
\Gamma_p(\langle \frac{p^{i}}{3}-\frac{ap^{i}}{q-1}\rangle) }{\Gamma_{p}(\langle\frac{p^{i}}{2}\rangle)\Gamma_{p}(\langle\frac{p^{i}}{3}\rangle)\Gamma_{p} (\langle\frac{2p^{i}}{3})\rangle} \\
&\hspace{.4cm}\times \Gamma_p (\langle \frac{2p^{i}}{3}-\frac{ap^{i}}{q-1}\rangle)\Gamma_p (\langle p^{i}-\frac{ap^{i}}{q-1}\rangle)\Gamma_p (\langle\frac{ap^{i}}{q-1}\rangle)- \frac{1}{q(q-1)}\notag,
\end{align*}
where $s_{i,a}=-\left\lfloor\langle\frac{p^i}{3}\rangle-\frac{ap^i}{q-1}\right\rfloor-\left\lfloor\langle\frac{2p^i}{3}\rangle-\frac{ap^i}{q-1}\right\rfloor-\left\lfloor\langle\frac{p^i}{2}\rangle+\frac{ap^i}{q-1}\right\rfloor-\left\lfloor\frac{ap^i}{q-1}\right\rfloor$. \\
Employing \eqref{eq-12}, we find that
\begin{align*}
A&= -\frac{1}{q-1}\sum_{a=1}^{q-2}\overline{\omega}^a\left(\frac{1}{x}\right)\prod_{i=0}^{r-1}(-p)^{s_{i,a}}\frac{\Gamma_p(\langle\frac{ap^{i}}{q-1}\rangle) \Gamma_p(\langle\frac{p^{i}}{2}+\frac{ap^{i}}{q-1}\rangle) 
\Gamma_p(\langle \frac{p^{i}}{3}-\frac{ap^{i}}{q-1}\rangle)}{\Gamma_{p}(\langle\frac{p^{i}}{2}\rangle)\Gamma_{p}(\langle\frac{p^{i}}{3}\rangle)\Gamma_{p} (\langle\frac{2p^{i}}{3})\rangle}\nonumber\\
&\hspace{.4cm}\times \Gamma_p (\langle \frac{2p^{i}}{3}-\frac{ap^{i}}{q-1}\rangle)- \frac{1}{q(q-1)}\\
&={ _{2}G}_{2}\left[\begin{array}{cc}
           \frac{1}{3}, & \frac{2}{3} \\
           0, & \frac{1}{2}
         \end{array}|\frac{1}{x}\right]_q -\frac{1}{q(q-1)}+\frac{1}{q-1}  \\
         &={ _{2}G}_{2}\left[\begin{array}{cc}
           \frac{1}{3}, & \frac{2}{3} \\
           0, & \frac{1}{2}
         \end{array}|\frac{1}{x}\right]_q +\frac{1}{q}. 
\end{align*}
This completes the proof of the proposition.
\end{proof}
\begin{proposition}\label{prop-2}
Let $p\geq 5$ be a prime and $q=p^r, r\geq 1$. For $x\in\mathbb{F}_q^{\times}$, we have
\begin{align}
\frac{1}{q(q-1)}\sum_{\chi\in\widehat{\mathbb{F}_q^{\times}}}
{g(\chi)}{g(\chi^{2})}{g(\overline{\chi}^{3})}\chi\left(\frac{-27}{4x}\right)=\frac{1}{q}+\varphi(3x)\cdot  {_2G}_{2}\left[\begin{array}{cc}
           \frac{1}{6}, & \frac{5}{6} \\
           0, & \frac{1}{2}
         \end{array}|\frac{1}{x}\right]_q.\notag\end{align}
\end{proposition}
\begin{proof}
As in Proposition \ref{prop-1}, we have
\begin{align*}
A= \frac{1}{q(q-1)} \sum_{a=0}^{q-2}
{g(T^{a})}{g(T^{2a})}{g(T^{-3a})}T^{a}\left(\frac{-27}{4x}\right).
\end{align*}
Replacing $a$ by $a- \frac{q-1}{2}$, we obtain
\begin{align}\label{eq-31}
A&=  \frac{1}{q(q-1)} \sum_{a=0}^{q-2}
{g(T^{a}\varphi)}{g(T^{2a})}{g(T^{-3a}\varphi)}T^{a}\left(\frac{-27}{4x}\right)\varphi(-3x).
\end{align}
Using Davenport-Hasse relation for $m=2$, $\psi = T^{a}$ and $m=2$, $\psi=T^{-3a}$, we have
\begin{align}
g(T^{a}\varphi)&=\frac{g(T^{2a})g(\varphi)T^{a}(2^{-2})}{g(T^{a})},\notag \\
g(T^{-3a}\varphi)&=\frac{g(T^{-6a})g(\varphi)T^{-3a}(2^{-2})}{g(T^{-3a})}.\notag
\end{align}
Substituting these values in \eqref{eq-31}, we deduce that
\begin{align}
A&=\frac{\varphi(-3x)}{q(q-1)} \sum_{a=0}^{q-2} \frac{g^{2}(T^{2a})g(\varphi)T^{a}(2^{-2})g(T^{-6a})g(\varphi)T^{-3a}(2^{-2})}{g(T^{a})g(T^{-3a})}T^{a}\left(\frac{-27}{4x}\right).\notag
\end{align}
Lemma \ref{lemma2_1} yields
\begin{align}
&A=\frac{\varphi(3x)}{q-1}\sum_{a=0}^{q-2} \frac{g^{2}(T^{2a})g(T^{-6a})}{g(T^{a})g(T^{-3a})}T^{a}\left(\frac{-27\times 4}{x}\right).\notag
\end{align}
Replacing $T$ by $\overline{\omega}$ and then applying Gross-Koblitz formula, we obtain
\begin{align}
&A=-\frac{\varphi(3x)}{q-1}\sum_{a=0}^{q-2} \overline{\omega}^{a}\left(\frac{-27\times4}{x}\right) 
\prod_{i=0}^{r-1} p^{\beta_{i,a}}\frac{\Gamma_p^{2} (\langle\frac{2ap^{i}}{q-1}\rangle)\Gamma_p (\langle\frac{-6ap^{i}}{q-1}\rangle)}{\Gamma_p(\langle\frac{ap^{i}}{q-1}\rangle)\Gamma_p (\langle\frac{-3ap^{i}}{q-1}\rangle)},\notag
\end{align}
where $\beta_{i,a}=2\langle\frac{2ap^{i}}{q-1}\rangle +\langle\frac{-6ap^{i}}{q-1}\rangle-\langle\frac{-3ap^{i}}{q-1}\rangle-\langle\frac{ap^{i}}{q-1}\rangle$.
Using Lemma \ref{lemma3_1} for $t=3$ and $t=6$ and Lemma \ref{lemma3_2} for $t=2$, and then employing \eqref{eq-13}, we deduce that
\begin{align*}
A& =-\frac{\varphi(3x)}{q-1}\sum_{a=0}^{q-2} \overline{\omega}^{a}\left(\frac{-1}{x}\right) \prod_{i=0}^{r-1} p^{\beta_{i,a}}
\frac{\Gamma_p (\langle\frac{p^{i}}{6}-\frac{ap^{i}}{q-1}\rangle)\Gamma_p(\langle\frac{5p^{i}}{6}-\frac{ap^{i}}{q-1}\rangle)\Gamma_p(\langle\frac{ap^{i}}{q-1}\rangle)}{\Gamma_p (\langle\frac{p^{i}}{2}\rangle)\Gamma_p (\langle\frac{p^{i}}{6}\rangle)\Gamma_p (\langle\frac{5p^{i}}{6}\rangle)}\\
&\hspace{4.5cm}\times\frac{\Gamma_p (\langle\frac{p^{i}}{2}+\frac{ap^{i}}{q-1}\rangle)\Gamma_p (\langle\frac{p^{i}}{2}+\frac{ap^{i}}{q-1}\rangle)\Gamma_p (\langle\frac{p^{i}}{2}-\frac{ap^{i}}{q-1}\rangle)}{\Gamma_p^{2} (\langle\frac{p^{i}}{2}\rangle)}\\
&=-\frac{\varphi(3x)}{q-1}\sum_{a=0,a\neq \frac{q-1}{2}}^{q-2} \overline{\omega}^{a}\left(\frac{1}{x}\right) \prod_{i=0}^{r-1} 
p^{\beta_{i,a}}\frac{\Gamma_p (\langle\frac{p^{i}}{6}-\frac{ap^{i}}{q-1}\rangle)\Gamma_p(\langle\frac{5p^{i}}{6}-\frac{ap^{i}}{q-1}\rangle)\Gamma_p(\langle\frac{ap^{i}}{q-1}\rangle)}{\Gamma_p(\langle\frac{p^{i}}{2}\rangle)
\Gamma_p(\langle\frac{p^{i}}{6}\rangle)\Gamma_p(\langle\frac{5p^{i}}{6}\rangle)}\\
&\hspace{3.5cm}\times \Gamma_p(\langle\frac{p^{i}}{2}+\frac{ap^{i}}{q-1}\rangle)-\frac{\varphi(3)}{q(q-1)}\prod_{i=0}^{r-1}\frac{\Gamma_p (\langle\frac{p^{i}}{3}\rangle)\Gamma_p (\langle\frac{2p^{i}}{3}\rangle)}{\Gamma_p (\langle\frac{p^{i}}{6}\rangle)\Gamma_p (\langle\frac{5p^{i}}{6}\rangle)}.
\end{align*}
Putting $a=\frac{q-1}{2}$ and $t=3$ in Lemma \ref{lemma3_1} yields 
\begin{align}\label{new-eqn1}
\prod_{i=0}^{r-1}\frac{\Gamma_p (\langle\frac{p^{i}}{3}\rangle)\Gamma_p (\langle\frac{2p^{i}}{3}\rangle)}{\Gamma_p (\langle\frac{p^{i}}{6}\rangle)\Gamma_p (\langle\frac{5p^{i}}{6}\rangle)} =\varphi(3).
\end{align}
Using \eqref{new-eqn1} and Lemma \ref{lemma-3.2}, we obtain
\begin{align*}
A&=-\frac{\varphi(3x)}{q-1} \sum_{a=0,a\neq \frac{q-1}{2}}^{q-2} \overline{\omega}^{a}\left(\frac{1}{x}\right) \prod_{i=0}^{r-1} (-p)^{u_{i,a}}\frac{\Gamma_p(\langle\frac{p^{i}}{6}-\frac{ap^{i}}{q-1}\rangle)
\Gamma_p(\langle\frac{5p^{i}}{6}-\frac{ap^{i}}{q-1}\rangle)}{\Gamma_p(\langle\frac{p^{i}}{2}\rangle)\Gamma_p(\langle\frac{p^{i}}{6}\rangle)\Gamma_p(\langle\frac{5 p^{i}}{6}\rangle)}\\
&\hspace{4cm}\times \Gamma_p(\langle\frac{ap^{i}}{q-1}\rangle)\Gamma_p(\langle\frac{p^{i}}{2}+\frac{ap^{i}}{q-1}\rangle)-\frac{1}{q(q-1)},
\end{align*}
where $u_{i,a}=-\left\lfloor\langle\frac{p^i}{6}\rangle-\frac{ap^i}{q-1}\right\rfloor-\left\lfloor\langle\frac{5p^i}{6}\rangle-\frac{ap^i}{q-1}\right\rfloor-\left\lfloor\langle\frac{p^i}{2}\rangle+\frac{ap^i}{q-1}\right\rfloor-\left\lfloor\frac{ap^i}{q-1}\right\rfloor.$
Adding and subtracting the term for $a=\frac{q-1}{2}$, we deduce that 
\begin{align}
&A=\frac{1}{q}+\varphi(3x)\cdot {_{2}G}_{2}\left[\begin{array}{cc}
         \frac{1}{6}, & \frac{5}{6} \\
          0, & \frac{1}{2}
        \end{array}|\frac{1}{x}\right]_q.\notag
\end{align}
This completes the proof of the proposition.
\end{proof}
We prove Theorem \ref{MT-1} and Theorem \ref{MT-3} by using Proposition \ref{prop-1} and Proposition \ref{prop-2}. We first prove Theorem \ref{MT-3}.
\begin{proof}[Proof of Theorem \ref{MT-3}] We have $x\in \mathbb{F}_q$ and $x\neq 0, 1$. From Proposition \ref{prop-1}, we have
\begin{align}\label{new-eqn2}
A=\frac{1}{q(q-1)}\sum_{\chi\in\widehat{\mathbb{F}_q^{\times}}}
{g(\chi)}{g(\chi^{2})}{g(\overline{\chi}^{3})}\chi\left(\frac{-27}{4x}\right)=\frac{1}{q}+{_{2}G}_{2}\left[\begin{array}{cc}
\frac{1}{3}, & \frac{2}{3} \\
0, & \frac{1}{2}
\end{array}|\frac{1}{x}\right]_q.
\end{align}	
Also,  from Proposition \ref{prop-2}, we have
\begin{align}\label{new-eqn10}
A=\frac{1}{q}+\varphi(3x)\cdot {_{2}G}_{2}\left[\begin{array}{cc}
\frac{1}{6}, & \frac{5}{6} \\
0, & \frac{1}{2}
\end{array}|\frac{1}{x}\right]_q.
\end{align}	
Now, 
\begin{align*}
	A&=\frac{1}{q(q-1)}\sum_{\chi\in\widehat{\mathbb{F}_q^{\times}}}
	{g(\chi)}{g(\chi^{2})}{g(\overline{\chi}^{3})}\chi\left(\frac{-27}{4x}\right)\\
	&=\frac{1}{q(q-1)} \sum_{\chi\in\widehat{\mathbb{F}_q^{\times}}}
	{g(\chi)g(\overline{\chi})}\frac{{g(\overline{\chi}^{2})}{g(\chi^{3})}}{g(\chi)}\overline{\chi}\left(\frac{-27}{4x}\right).
\end{align*}
By using Lemma \ref{lemma2_1} and Lemma \ref{lemma2_2}, we obtain
\begin{align}\label{neweq-14}
A&= \frac{1}{q-1}\sum_{\chi\in\widehat{\mathbb{F}_{q}^{\times}}}J(\overline{\chi}^{2}, \chi^{3})\overline{\chi}\left(\frac{27}{4x}\right)-1-\frac{J(\varepsilon,\varepsilon)}{q}+\frac{q-1}{q}. 
\end{align}
Using \eqref{eq-0} and \eqref{eq-4}, we have $J(\varepsilon,\varepsilon)=q-2$. Putting the value of $J(\varepsilon,\varepsilon)$ in \eqref{neweq-14} and then using \eqref{eq-1}, we obtain
\begin{align}
A&=\frac{1}{q} -1 +\frac{1}{q-1} \sum_{\chi\in\widehat{\mathbb{F}_q^{\times}}}\overline{\chi}\left(\frac{27}{4x}\right)J(\overline{\chi}^{2}, \chi^{3})\notag\\
&=\frac{1}{q} -1 +\frac{1}{q-1} \sum_{\chi\in\widehat{\mathbb{F}_q^{\times}}}\overline{\chi}\left(\frac{27}{4x}\right)J(\overline{\chi}^{2}, \overline{\chi})\notag\\
&=\frac{1}{q} -1 +\frac{1}{q-1}\sum_{y\in \mathbb{F}_{q}} \sum_{\chi\in\widehat{\mathbb{F}_q^{\times}}}\overline{\chi}\left(\frac{27y^{2} (1-y)}{4x}\right).\label{eq-32}
\end{align}
By using \eqref{eq-2}, we know that the inner summation in \eqref{eq-32} gives non zero value only if the cubic equation $27y^{2}(1-y)-4x=0$ has a solution in $\mathbb{F}_{q}$. It is well-known that a cubic polynomial has exactly one root in 
$\mathbb{F}_{q}$ if and only if its discriminant is a non square in $\mathbb{F}_{q}$. The  discriminant of the polynomial $27y^{2}(1-y)-4x$ is equal to $16\times 27^{3}\times x(1-x)$. Hence, from \eqref{eq-32} we deduce that
$A=\frac{1}{q}$ if and only if $\varphi(3x(1-x))=-1$. Using \eqref{new-eqn2} and \eqref{new-eqn10}, we complete the proof. 
\end{proof}
\begin{remark} We note that Proposition \ref{prop-2} can also be proved using \cite[Theorem 1.5]{BS4}. The proof goes as follows. For $x\in \mathbb{F}_q^{\times}$, let $f(y)=y^3-y^2+\frac{4x}{27}$. Then, \cite[Theorem 1.5]{BS4} yields
\begin{align}\label{neq-eqn11}
 \#\{a\in \mathbb{F}_q: f(a)=0\}=1+\varphi(3x)\cdot {_{2}G}_{2}\left[\begin{array}{cc}
0, & \frac{1}{2} \\
\frac{1}{6}, & \frac{5}{6}
\end{array}|x\right]_q.
\end{align}
By definition, we have 
\begin{align*}
	&_2G_2\left[\begin{array}{cccc}
	0, & \frac{1}{2} \\
\frac{1}{6}, & \frac{5}{6}
\end{array}|x\right]_q\\
	&\hspace{.5cm}=\frac{-1}{q-1}\sum_{a=0}^{q-2}\overline{\omega}^a(x)
	\prod\limits_{i=0}^{r-1}(-p)^{-\lfloor -\frac{ap^i}{q-1} \rfloor -\lfloor\langle -\frac{p^i}{6} \rangle +\frac{ap^i}{q-1}\rfloor-\lfloor \langle \frac{p^i}{2} \rangle-\frac{ap^i}{q-1} \rfloor -\lfloor\langle -\frac{5p^i}{6} \rangle +\frac{ap^i}{q-1}\rfloor}\notag\\
	&\hspace{.5cm} \times \frac{\Gamma_p(\langle (-\frac{a}{q-1})p^i\rangle)\Gamma_p(\langle (-\frac{1}{6}+\frac{a}{q-1})p^i \rangle)\Gamma_p(\langle (\frac{1}{2}-\frac{a}{q-1})p^i\rangle)\Gamma_p(\langle (-\frac{5}{6}+\frac{a}{q-1})p^i \rangle)}
	{\Gamma_p(\langle -\frac{p^i}{6} \rangle)\Gamma_p(\langle \frac{p^i}{2} \rangle)\Gamma_p(\langle -\frac{5p^i}{6} \rangle)}.
	\end{align*}
Replacing $a$ by $-a$, we deduce that 
\begin{align}\label{new-eqn12}
	&{_2}G_2\left[\begin{array}{cccc}
	0, & \frac{1}{2} \\
\frac{1}{6}, & \frac{5}{6}
\end{array}|x\right]_q= {_2}G_2\left[\begin{array}{cccc}
	\frac{1}{6}, & \frac{5}{6}\\
	0, & \frac{1}{2}
\end{array}|\frac{1}{x}\right]_q.
\end{align}
Now, combining \eqref{eq-32}, \eqref{neq-eqn11}, and \eqref{new-eqn12}, we obtain the required identity of Proposition \ref{prop-2}.
\end{remark}
We are now ready to prove Theorem \ref{MT-1}.
\begin{proof}[Proof of Theorem \ref{MT-1}]
If $x\neq 0$, then combining Proposition \ref{prop-1} and Proposition \ref{prop-2}, we obtain
\begin{align}\label{new-eqn13}
{ _{2}G}_{2}\left[\begin{array}{cc}
           \frac{1}{3}, & \frac{2}{3} \\
           0, & \frac{1}{2}
         \end{array}|\frac{1}{x}\right]_q =\varphi(3x)\cdot {_{2}G}_{2}\left[\begin{array}{cc}
           \frac{1}{6}, & \frac{5}{6} \\         
           0, & \frac{1}{2}
         \end{array}|\frac{1}{x}\right]_q.
\end{align}
By taking $x=1$, we obtain \eqref{eq2-MT-1}. 
If $x\neq 0, 1$, then $\varphi(3x(1-x))=\pm 1$. If $\varphi(3x(1-x))= 1$, then $\varphi(3x)=\varphi(1-x)$, and we readily obtain \eqref{eq1-MT-1} from \eqref{new-eqn13}. If $\varphi(3x(1-x))= -1$, then we obtain \eqref{eq1-MT-1} by using Theorem \ref{MT-3}. This completes the proof of the theorem.
\end{proof}
\begin{proof}[Proof of Theorem \ref{MT-2}]
In \cite[Theorem 2.1]{ahlgren1}, Ahlgren et al.  proved that if $\lambda\in \mathbb{F}_{q} $ such that $\lambda \neq 0, -1$, then 
\begin{align}\label{eq-33}
A(\lambda,q)=\varphi(\lambda+1)(a(\lambda,q)^{2} -q),
\end{align}
where $a(\lambda,q)$ and $A(\lambda,q)$ are defined as
\begin{align*}
a(\lambda,q):&= \sum_{x\in\mathbb{F}_{q}} \varphi\left((x-1)\left(x^{2}-\frac{1}{\lambda +1}\right)\right),\\
A(\lambda,q) :&= \sum_{x,y\in\mathbb{F}_{q}} \varphi(xy(x+1)(y+1)(x+\lambda y)).
\end{align*}
Let $\displaystyle h(\lambda):=\frac{1}{q-1} \sum_{\chi\in\widehat{\mathbb{F}_q^{\times}}}\chi\left(\frac{1}{\lambda}\right)J(\overline{\chi}\varphi, \chi)^{3}$. Then using \eqref{eq-1}, we have 
\begin{align}
h(\lambda)&=\frac{1}{q-1} \sum_{\chi\in\widehat{\mathbb{F}_q^{\times}}}\chi\left(\frac{1}{\lambda}\right)\chi(-1)J(\chi,\varphi)^{3}\notag\\
&=\frac{1}{q-1}\sum_{\chi\in\widehat{\mathbb{F}_q^{\times}}}\chi\left(\frac{-xyz}{\lambda}\right)\sum_{x,y,z\in \mathbb{F}_q^{\times}}\varphi(1-x)\varphi(1-y)\varphi(1-z).\notag
\end{align}
Using \eqref{eq-2}, we obtain 
\begin{align}\label{new-eqn3}
h(\lambda)&=\sum_{x,y\in \mathbb{F}_q^{\times}}\varphi(1-x)\varphi(1-y)\varphi\left(1+\frac{\lambda}{xy}\right)\notag\\
&=\sum_{x,y\in \mathbb{F}_q^{\times}}\varphi((1+x)(1+y)xy(xy+\lambda))\notag\\
&=\sum_{x,y\in \mathbb{F}_q^{\times}}\varphi((1+x)(1+y)xy(x+y\lambda))\notag\\
&=A(\lambda,q).
\end{align}
Now using Lemma \ref{lemma2_2} in $h(\lambda)$ and observing that  $\overline{\omega}^{a}\left(-1\right)=(-1)^{a} $ and\\ $3\langle\frac{p^i}{2} -\frac{ap^{i}}{q-1} \rangle +3\langle\frac{ap^{i}}{q-1}\rangle - 3\langle\frac{p^{i}}{2}\rangle 
=-3\lfloor \langle\frac{p^i}{2}\rangle -\frac{ap^{i}}{q-1} \rfloor -3\lfloor\frac{ap^{i}}{q-1}\rfloor$, we have
\begin{align}\label{new-eqn4}
h(\lambda)&=\frac{1}{q-1} \sum_{\chi\in\widehat{\mathbb{F}_q^{\times}}} \chi\left(\frac{1}{\lambda}\right)\frac{g^{3}(\overline{\chi\varphi})g^{3}(\chi)}{g^{3}(\varphi)}\notag\\
&=-\frac{1}{q-1} \sum_{a=0}^{q-2} \overline{\omega}^{a}\left(\frac{1}{\lambda}\right) \prod_{i=0}^{r-1} (-p)^{3\langle\frac{p^i}{2} -\frac{ap^{i}}{q-1} \rangle +3\langle\frac{ap^{i}}{q-1}\rangle - 3\langle\frac{p^{i}}{2}\rangle}\notag\\ 
&\hspace{4cm}\times \frac{\Gamma_p^{3}(\langle\frac{p^{i}}{2}-\frac{ap^{i}}{q-1}\rangle)\Gamma_p^{3}(\langle\frac{ap^{i}}{q-1}\rangle)}{\Gamma_p^{3} (\langle\frac{p^{i}}{2}\rangle)}\notag\\
&=-\frac{1}{q-1} \sum_{a=0}^{q-2} \overline{\omega}^{a}\left(\frac{-1}{\lambda}\right)\overline{\omega}^{a}\left(-1\right) \prod_{i=0}^{r-1} (-p)^{-3\lfloor \langle\frac{p^i}{2}\rangle -\frac{ap^{i}}{q-1} \rfloor -3\lfloor\frac{ap^{i}}{q-1}\rfloor}\notag\\ 
&\hspace{4cm}\times\frac{\Gamma_p^{3}(\langle\frac{p^{i}}{2}-\frac{ap^{i}}{q-1}\rangle)\Gamma_p^{3}(\langle\frac{ap^{i}}{q-1}\rangle)}{\Gamma_p^{3}(\langle\frac{p^{i}}{2}\rangle)}\notag\\
&={_{3}G}_{3}\left[\begin{array}{ccc}
\frac{1}{2} & \frac{1}{2} & \frac{1}{2} \\
0 & 0 & 0
\end{array}|\frac{-1}{\lambda}\right]_q.
\end{align}
Combining \eqref{new-eqn3} and \eqref{new-eqn4}, we have 
\begin{align*}
A(\lambda, q)={_{3}G}_{3}\left[\begin{array}{ccc}
\frac{1}{2} & \frac{1}{2} & \frac{1}{2} \\
0 & 0 & 0
\end{array}|\frac{-1}{\lambda}\right]_q.
\end{align*}
We now consider a character sum which is related to $a(\lambda,q)$.
\begin{align}
B:&=\frac{q^2\varphi(-2)}{q-1}\sum_{\chi\in\widehat{\mathbb{F}_q^{\times}}}{\varphi\chi^2\choose \chi}{\varphi\chi\choose \chi}\chi\left(\frac{\lambda}{4(\lambda+1)}\right).\notag
\end{align}
Using \eqref{eq-0}, we have
\begin{align}
B&= \frac{\varphi(-2)}{q-1}\sum_{\chi\in\widehat{\mathbb{F}_q^{\times}}}{J(\varphi\chi^2,\overline{\chi})}{J(\varphi\chi,\overline{\chi})}\chi\left(\frac{\lambda}{4(\lambda+1)}\right)\notag\\
&=\frac{\varphi(-2)}{q-1}\sum_{\chi\in\widehat{\mathbb{F}_q^{\times}}}{J(\overline{\chi},\overline{\varphi\chi})}{J(\overline{\chi},\overline{\varphi})}\chi\left(\frac{\lambda}{4(\lambda+1)}\right).\notag
\end{align}
By \cite[Lemma 2.2]{ahlgren1}, we have
\begin{align}\label{eq-35}
B=-\varphi\left(\frac{2\lambda}{\lambda+1}\right)-\varphi(-1)a(\lambda,q).
\end{align}
Employing \cite[Proposition 1]{BS3}, we have
\begin{align}\label{eq-36}
B&= -\varphi \left(\frac{2\lambda}{\lambda+1}\right) -\varphi(-2)\cdot {_2G}_{2}\left[\begin{array}{cc}
\frac{1}{4}, & \frac{3}{4} \\
0, & 0
\end{array}|\frac{\lambda+1}{\lambda}\right]_q.
\end{align}
Combining \eqref{eq-35} and \eqref{eq-36}, we have
\begin{align*}
a(\lambda,q)&= \varphi(2)\cdot {_2G}_{2}\left[\begin{array}{cc}
\frac{1}{4}, & \frac{3}{4} \\
0, & 0
\end{array}|\frac{\lambda+1}{\lambda}\right]_q.
\end{align*}
Substituting the values of $a(\lambda,q)$ and $A(\lambda,q)$ in \eqref{eq-33}, we obtain
\begin{align}
{_3G}_{3}\left[\begin{array}{ccc}
\frac{1}{2}, & \frac{1}{2}, & \frac{1}{2} \\
0, & 0, & 0
\end{array}|\frac{-1}{\lambda}
\right]_q &=\varphi(1+\lambda)\cdot
{_2G}_{2}\left[\begin{array}{cc}
\frac{1}{4}, & \frac{3}{4} \\
0, & 0
\end{array}|\frac{\lambda+1}{\lambda}\right]_q^2-q\cdot\varphi(1+\lambda).\notag
\end{align}
Putting $\lambda=-x$ we obtain the required identity.
\end{proof}

\end{document}